\documentclass{article}

\usepackage{amsmath,amsthm,amssymb,amsfonts}
\usepackage[colorlinks=true,urlcolor=blue,citecolor=blue,linkcolor=blue]{hyperref}

\newtheorem{theorem}{Theorem}[section]
\newtheorem{remark}[theorem]{Remark}
\newtheorem{lemma}[theorem]{Lemma}
\newtheorem{corollary}[theorem]{Corollary}
\newtheorem{assumption}[theorem]{Assumption}
\newtheorem{myalg}{Algorithm}

\renewcommand{\themyalg}{\textbf{\Alph{myalg}}}

\def \R{\mathbb{R}}

\def \N{\mathbb{N}}
\def \eps{\varepsilon}
\def \Uad{{U_{\text{ad}}}}

\newcommand{\dx}{\,\text{\rm{}d}x}

\numberwithin{equation}{section}

\makeatletter
\renewcommand*{\env@matrix}[1][*\c@MaxMatrixCols c]{%
  \hskip -\arraycolsep
  \let\@ifnextchar\new@ifnextchar
  \array{#1}}
\makeatother

\newenvironment{keywords}{\par\addvspace{11pt}\noindent{{\bf{Keywords:}} }\ignorespaces}{\par\addvspace{26pt plus 4pt}}
\newenvironment{classcode}{\par\addvspace{11pt}\noindent{{\bf{AMS Subject Classification:}} }\ignorespaces}{\par\addvspace{26pt plus 4pt}}

\begin{document}

\title{An iterative Bregman regularization method for optimal control problems with inequality constraints}

\author{
Frank P\"orner\footnote{Department of Mathematics,  University of W\"urzburg, Emil-Fischer-Str.\@ 40, 97074 W\"urzburg, Germany, E-mail: frank.poerner@mathematik.uni-wuerzburg.de}
,
Daniel Wachsmuth\footnote{Department of Mathematics, University of W\"urzburg, Emil-Fischer-Str.\@ 40, 97074 W\"urzburg, Germany, E-mail: daniel.wachsmuth@mathematik.uni-wuerzburg.de}
}
\date{}
\maketitle

\begin{abstract}
We study an iterative regularization method of optimal control problems with control constraints.
The regularization method is based on generalized Bregman distances. We provide convergence results
under a combination of a source condition and a regularity condition on the active sets.
We do not assume attainability of the desired state.
Furthermore, a-priori regularization error estimates are obtained.
\end{abstract}


\begin{classcode}
49K20, %
49N45, %
65K10  %
\end{classcode}

\section{Introduction}
In this article we consider optimization problems of the following form
\begin{equation}\label{eq:2_main_problem}\tag{\textbf{P}}
 \begin{split}
    \text{Minimize} &\quad \frac{1}{2}\|Su - z\|_Y^2  \\
    \text{such that} &\quad u_a \leq u \leq u_b \quad \text{a.e. in } \Omega,
\end{split}
\end{equation}
which can be interpreted both as an optimal control problem or as an inverse problem. Here $\Omega \subseteq \R^n$, $n \geq 1$ is a bounded, measurable set, $Y$ a Hilbert space, $z \in Y$ a given function. The operator $S : L^2(\Omega) \to Y$ is linear and continuous. The inequality constraints are prescribed on the set $\Omega$. We assume $u_a,u_b \in L^\infty(\Omega)$.
Here, we have in mind to choose $S$ to be the solution operator of a linear partial differential equation. In many situations
the operator $S$ is compact or has non-closed range, which makes \eqref{eq:2_main_problem} ill-posed.

In an optimal control setting, the unknown $u$ is the control and the constraints are limitations arising from the underlying physical problem, i.e., temperature restriction of a heat source. The function $z$ is the desired state, and we search for $u$ such that $Su$ is as close to $z$ as possible with respect to the norm in $Y$.
Here, the interesting  situation is, when $z$ cannot be reached due to the presence of the control constraints (non-attainability).
If
\eqref{eq:2_main_problem} is interpreted as an inverse problem, the unknown $u$ represents some data to be reconstructed from the measurement $z$. Here the inequality constraints reflect a-priori information of the unknown $u$.

A well-known regularization method is the Tikhonov regularization with some positive regularization parameter $\alpha > 0$. The regularized problem is given by:
\begin{equation*}
 \begin{split}
    \text{Minimize} &\quad \frac{1}{2}\|Su - z\|_Y^2 + \frac{\alpha}{2}\|u\|_{L^2(\Omega)}^2  \\
    \text{such that} &\quad u_a \leq u \leq u_b \quad \text{a.e. in } \Omega.
\end{split}
\end{equation*}
The additional term can be interpreted as control costs. This method is well understood in regard to convergence for $\alpha \to 0$, perturbed data, and numerical approximations, see e.g.,
\cite{engl1996,troelsch2010,wachsmuth2013,wachsmuth2011,wachsmuth2011b}.
However, for $\alpha$ tending to zero the Tikhonov regularized problem becomes increasingly ill-conditioned.

An alternative is  the proximal point algorithm (PPM) introduced by Martinez~\cite{Martinet1970} and developed by Rockafellar \cite{Rockafellar1976}. Given an iterate $u_k$, the next iterate $u_{k+1}$ is determined by solving
\begin{equation*}
 \begin{split}
    \text{Minimize} &\quad \frac{1}{2}\|Su - z\|_Y^2 + \alpha_{k+1}\|u - u_k\|_{L^2(\Omega)}^2  \\
    \text{such that} &\quad u_a \leq u \leq u_b \quad \text{a.e. in } \Omega.
\end{split}
\end{equation*}
Due to the self-canceling effect of the regularization term, there is hope to obtain a convergent method without
the requirement that the regularization parameters $\alpha_k$ tend to zero.
However, in general PPM is not strongly convergent due to the example given by  G\"uler \cite{gueler1991},
which exhibits weakly converging but not strongly converging iterates, see also \cite{kaplan1994stable}.
An application of this method to optimal control problems is investigated in \cite{rotin2004}.
There exists strongly convergent  modifications of PPM, see e.g., \cite{Sabach2009,Sabach2010,Solodov00}.
Here, it is an open question how to transfer these methods to our problem while exploiting its particular structure.

In the inverse problems community this method is known under the name iterated Tikhonov regularization \cite{engl1996,hankegroetsch98}.
Under the attainability assumption, that is, $z$ is in the range of $S$, convergence can be proven. If one assumes in addition a so-called source condition,
then convergence rates can be derived.
While the PPM and thus the iterated Tikhonov method allow to proof beautiful monotonicity properties, we were
not able to show strong convergence under conditions adapted to our situation (control constraints and non-attainability).

In order to overcome this difficulty, we investigated the Bregman iterative regularization technique, where the Hilbert space norm in the regularization term
is replaced by a Bregman distance \cite{Bregman67}.
There, the iterate $u_{k+1}$ is given by the solution of
\begin{equation*}
\text{Minimize} \quad \frac{1}{2}\|Su-z\|_Y^2 + \alpha_{k+1}  D^{\lambda_{k}}(u,u_{k}),
\end{equation*}
where $D^\lambda(u,v) = J(u) - J(v) - (u-v,\lambda)$ is called the (generalized) Bregman distance associated to a regularization function $J$ with subgradient $\lambda \in \partial J(v)$.
This iteration method was used first in \cite{burger2007,osher2005}
applied to an image restoration problem with $J$ being the total variation.
Note that for the special choice $J(u) = \frac{1}{2}\|u\|_{L^2(\Omega)}^2$ the PPM algorithm is obtained.

We choose to incorporate the control constraint into the regularization functional, resulting in
\[
J(u) := \frac{1}{2}\|u\|^2 + I_\Uad(u),
\]
where $\Uad = \{u \in L^2(\Omega): \; u_a \leq u \leq u_b\}$, and $I$ is the indicator function of convex analysis.
While at first sight the incorporation of $I_\Uad$ into the Bregman regularization functional together with the explicit
control constraint $u\in U_{ad}$ seems to be redundant, this choice proved advantageous for the convergence analysis.

In order to prove convergence, in  \cite{burger2007} a source condition is imposed.
Moreover, the analysis there relied heavily on the attainability of $z$.
In this paper, we prove convergence and convergence rates without the attainability assumption.
To do so, the existing proof techniques had to be considerably extended.
Moreover, as argued in \cite{wachsmuth2011} a source condition is unlikely to hold in an optimal control setting if $z$ is not attainable, i.e., there is no feasible $u$ such that $z=Su$. In \cite{wachsmuth2013,wachsmuth2011b} a regularity assumption on the active set is used as suitable substitution of the source condition.
Here, the active set denotes the subset of $\Omega$, where the inequality constraints are active in the solution. However this assumption implies that the control constraints are active everywhere, and situations where the control constraints are inactive on a large part of $\Omega$ are not covered. To overcome this, in \cite{wachsmuth2011} both approaches are combined: A source condition is used on the part of the domain, where the inequality constraints are inactive, and a structural assumptions is used on the active sets.
We will use this combined assumption to prove convergence rates of the Bregman iteration.

In order to formulate the method, a recipe to choose the subgradient $\lambda$ has to be added.
We report on this choice in Section \ref{sec3}.
The convergence of the Bregman method is studied in Section \ref{sec4}.
Convergence rates under the assumption of a source condition are proven in Section \ref{sec:sc}.
The main result of the paper, the convergence under a combined source condition and regularity condition on the active sets
is Theorem \ref{thm:ASC_strong_convergence}, which can be found in Section \ref{sec44}.

\paragraph*{Notation.}
For elements $q \in L^2(\Omega)$, we denote the $L^2$-Norm by $\|q\| := \|q\|_{L^2(\Omega)}$. Furthermore $c$ is a generic constant, which may change from line to line, but is independent from the important variables, e.g. $k$.

\section{Problem setting and preliminary results}
Let $\Omega \subseteq \R^n$, $n\in \N$, be a bounded, measurable domain, $Y$ a Hilbert space, $S: L^2(\Omega) \to Y$ linear and continuous. We are interested in computing a solution to the minimization problem \eqref{eq:2_main_problem}.
Here, we assume $z \in Y$ and $u_a, u_b \in L^\infty(\Omega)$.
The functional to be minimized will be denoted by
\[
H(u) := \frac{1}{2}\|Su - z\|_Y^2
\]
and the set of admissible functions by
\[
\Uad := \{u \in L^2(\Omega): \; u_a \leq u \leq u_b\}.
\]
In addition we assume that $u_a\le u_b$ a.e.\@ on $\Omega$, which ensures that $\Uad$ is non-empty.

\subsection{Existence of solutions and optimality conditions}
\label{sec21}

Existence of solutions can be proven by classical arguments using the direct method of the calculus of variations.

\begin{theorem}\label{thm:existence_solutions}
Under the assumptions listed above the problem \eqref{eq:2_main_problem} has a solution.
If the operator $S$ is injective the solution is unique.
\end{theorem}

In the following, we denote by $u^\dagger \in \Uad$ a solution of \eqref{eq:2_main_problem}.
Note that due to the strict convexity of $H$ with respect to $Su$ the optimal state $y^\dagger:=Su^\dagger$ is uniquely defined.
In addition, we define the associated adjoint state $p^\dagger := S^\ast(z- Su^\dagger)$. We then have the following result.

\begin{theorem}\label{thm:relation_ud_pd}
We have the relation
$$u^\dagger (x)  \begin{cases}
= u_a(x) & \text{if} \quad p^\dagger(x) < 0 \\
\in \left[u_a(x), u_b(x)\right] & \text{if} \quad p^\dagger(x) = 0 \\
= u_b & \text{if} \quad p^\dagger(x) > 0 \\
\end{cases}$$
and the following variational inequality holds:
$$(-p^\dagger, u - u^\dagger) \geq 0, \quad \forall u \in \Uad.$$
\end{theorem}

This result shows that the solution $u^\dagger$ can be determined from $p^\dagger$ if the set $\{x: \ p^\dagger(x)\ne0\}$ has
zero measure. As a consequence, the problem \eqref{eq:2_main_problem} is uniquely solvable in this case.

\subsection{Bregman distance}
We want to apply the Bregman iteration with the regularization functional
$$J(u) := \frac{1}{2}\|u\|^2 + I_\Uad(u),$$
where $I_C$ denotes the indicator function of the set $C$. The Bregman distance for $J$ at $u,v \in L^2(\Omega)$ and $\lambda \in \partial J(v)$ is defined as
$$D^\lambda(u,v) := J(u) - J(v) - (u-v, \lambda).$$
Note that $\lambda = v + w$ with $w \in \partial I_\Uad(v)$, hence:
\begin{equation}\label{eq:Bregman_distance}
D^\lambda(u,v) = \frac{1}{2}\|u-v\|^2 + I_\Uad(u) - I_\Uad(v) - (u-v,w).
\end{equation}
Let us summarize the properties of $J$ and $D$:
\begin{lemma}\label{lemma:Bregman_nonegativ}
Let $C \subseteq L^2(\Omega)$ be non-empty, closed, and convex. The functional
$$J: L^2(\Omega) \to \mathbb{R}\cup \{ + \infty\}, \quad u \mapsto \frac{1}{2}\|u\|^2 + I_C(u)$$
is convex and nonnegative. Furthermore the Bregman distance
$$D^\lambda(u,v) := J(u) - J(v) - (u-v, \lambda), \quad \lambda \in \partial J(v)$$
is nonnegative and convex with respect to $u$.
\end{lemma}

The subgradient $\partial I_\Uad(v)$ is the normal cone of $\Uad$ at $v$, which can be characterized as:
$$\partial I_\Uad(v) = \left\{ w \in L^2(\Omega): \quad  w(x)      \begin{cases}
    \leq 0 & \text{if } v(x)= u_a(x) \\
    = 0 & \text{if } u_a(x) < v(x) < u_b(x) \\
    \geq 0 & \text{if } v(x) = u_b(x)
    \end{cases} \quad \right\}.$$
Hence, we have for the Bregman distance at $v\in \Uad$
\begin{align*}
D^{\lambda}(u,v) &= \frac{1}{2}\|u-v\|^2 + I_\Uad(u) \\
&\quad + \int\limits_{\{v=u_a\}} w ( u_a - u) \;\dx + \int\limits_{\{v=u_b\}} w ( u_b - u) \;\dx.
\end{align*}
where we abbreviated by $\{v = u_a\}$ the set $\{x \in \Omega: \; v(x) = u_a(x)\}$. We see that the Bregman distance adds
two parts that measures $u$ on sets where the control constraints are active for $v$.
Due to the properties of $w\in \partial I_\Uad(v)$ we obtain
\begin{equation}\label{eq202}
 D^{\lambda}(u,v) \ge \frac{1}{2}\|u-v\|^2 \quad \forall u,v\in \Uad, \ \lambda \in \partial J(v).
\end{equation}
Since the subgradient $\partial I_\Uad(v)$ is not a singleton in general,
the Bregman distance depends on the choice of the subgradient $w \in \partial I_\Uad(v)$.
In the algorithm described below we will derive a suitable choice for the subgradients $\lambda\in\partial J(u)$ and $w\in\partial I_\Uad(u)$.

\section{Bregman iteration}
\label{sec3}
To start our algorithm we need suitable starting values $u_0 \in \Uad$ and $\lambda_0 \in \partial J(u_0)$. We define $u_0$ to be the solution of the problem
$$\min\limits_{u \in L^2(\Omega)} J(u) = \frac{1}{2} \|u\|^2 + I_\Uad(u),$$
which yields $u_0 = P_\Uad(0)$. Furthermore this choice ensures $0 \in \partial J(u_0)$, so we simply set $\lambda_0 = 0$. Note that all of the following results can be extended to arbitrary $u_0 \in \Uad$ and general subgradients $\lambda_0 \in \partial J(u_0) \cap R(S^\ast)$.
The (prototypical) Bregman iteration is now defined as follows:
{
\renewcommand{\themyalg}{\textbf{A${}_0$}}
\begin{myalg}\label{alg:Min1}
Let $u_0 = P_\Uad(0)  \in \Uad$, $\lambda_0=0  \in \partial J(u_0)$ and $k=1$.
\begin{enumerate}
  \item Solve for $u_k$: \label{a0_start}
  \begin{equation}
\label{Min1}
\text{Minimize} \quad \frac{1}{2}\|Su-z\|_Y^2 + \alpha_{k}  D^{\lambda_{k-1}}(u,u_{k-1}).
\end{equation}
  \item Choose $\lambda_k \in \partial J(u_k)$.
  \item Set $k:=k+1$, go back to \ref{a0_start}.
\end{enumerate}
\end{myalg}
\setcounter{myalg}{0}
}
Here $(\alpha_k)_k$ is a bounded sequence of positive real numbers. If $u^\dagger$ is a solution of \eqref{eq:2_main_problem}, it satisfies $u^\dagger = P_\Uad \big( u^\dagger - \Theta S^\ast (Su^\dagger - z) \big)$ with $\Theta > 0$ arbitrary. Therefore a possible stopping criterion is given by (with $\eps > 0$)
$$\left\| u_k - P_\Uad \big( u_k - \Theta S^\ast (Su_k - z) \big) \right\| \leq \eps.$$
We introduce the quantity
$$\gamma_k := \sum\limits_{j=1}^k \frac{1}{\alpha_j}.$$
Since the sequence $\alpha_j$ is bounded we obtain
$$\lim\limits_{k \to \infty} \gamma_k^{-1} = 0.$$

In algorithm \ref{alg:Min1} it remains to specify how to choose the subgradient $\lambda_k$ for the next iteration. We will show that we can construct a new subgradient based on the iterates $u_1,...,u_k$. The following result motivates the construction of the subgradient. Moreover it shows that algorithm \ref{alg:Min1} is well-posed.
\begin{lemma}\label{lemma:opt_alg}
The problem \eqref{Min1} has a unique solution $u_k \in \Uad$ and there exists $w_k \in \partial I_\Uad(u_k)$ such that
$$S^\ast(Su_k - z) + \alpha_k( u_k - \lambda_{k-1} + w_k  ) = 0.$$
Moreover, the subgradient $\partial J(u_k)$ is non-empty.
\end{lemma}

\begin{proof}
The set of admissible functions $\Uad$ is nonempty, closed, convex, and bounded, hence weakly compact. Furthermore,
the function $J_k$ defined by
\[
J_k: L^2(\Omega) \to \R, \quad u \mapsto  \frac{1}{2}\|u-u_{k-1}\|^2 - (u-u_{k-1}, \lambda_{k-1})
\]
is continuous and convex, hence it is weakly lower semi-continuous.
It is easy to check that \eqref{Min1}
is equivalent to
\[
 \min\limits_{u \in \Uad} H(u) + \alpha_k J_k(u).
 \]
Since $H$ is convex, the function $H+\alpha_k J_k$ is convex and by the Weierstra{\ss} theorem (with respect to the weak topology) we get existence of minimizers. Since $\alpha_k \neq 0$ and $J_k$ is strictly convex, minimizers are also unique. By the first-order optimality condition for \eqref{Min1}
there exists  $w_k \in \partial I_\Uad (u_k)$ such that
$$S^\ast(Su_k - z) + \alpha_k( u_k - \lambda_{k-1} + w_k  ) = 0.$$
Clearly, it holds $\partial J(u_k)\ne\emptyset$.
\end{proof}

We have $\partial J(u_k) = u_k + \partial I_\Uad(u_k)$, so motivated by Lemma \ref{lemma:opt_alg} we set
\begin{equation}\label{def:subdiff}
\lambda_k := u_k + w_k = \frac{1}{\alpha_k} S^\ast(z-Su_k) + \lambda_{k-1} \in \partial J(u_k)
\end{equation}
An induction argument now yields the following result.
\begin{lemma}\label{lemma:Algo_well_posed}
Let the subgradients $\lambda_k \in \partial J(u_k)$ be chosen according to \eqref{def:subdiff}. Then it holds
$$\lambda_k = S^\ast \mu_k, \quad \mu_k := \sum\limits_{i=1}^k \frac{1}{\alpha_i} (z-Su_i).$$
\end{lemma}

With this choice of $\lambda_k$,
we see that the Bregman iteration \ref{alg:Min1} can be equivalently formulated as:
\begin{myalg}\label{alg:Min2}
Let $u_0 = P_\Uad(0)  \in \Uad$, $\mu_0 = 0$, $\lambda_0 = 0  \in \partial J(u_0)$ and $k=1$.
\begin{enumerate}
  \item Solve for $u_k$: \label{a_start}
  \begin{equation}
\label{Min1xx}\tag{A1}
\text{Minimize} \quad \frac{1}{2}\|Su-z\|_Y^2 + \alpha_{k}  D^{\lambda_{k-1}}(u,u_{k-1}).
\end{equation}
  \item Set $\mu_k := \sum\limits_{i=1}^k \frac{1}{\alpha_i} (z-Su_i)$ and $\lambda_k := S^\ast \mu_k$.
  \item Set $k:=k+1$, go back to \ref{a_start}.
\end{enumerate}

\end{myalg}

As argued in \cite{burger2007,osher2005}, algorithm \ref{alg:Min2} is equivalent to the following algorithm:
\begin{myalg}\label{alg:Min3}
Let $\mu_0 := 0$ and $k=1$.
\begin{enumerate}
  \item Solve for $u_k$: \label{b_start}
  \begin{align*}
  \text{Minimize} \quad &\frac{1}{2}\|Su-z - \alpha_{k}\mu_{k-1}\|_Y^2 + \frac{\alpha_{k}}{2} \|u\|^2\\
  \text{such that} \quad & u_k \in \Uad
  \end{align*}
  \item Set $\mu_k = \cfrac{1}{\alpha_{k}}(z-Su_k) + \mu_{k-1}$.
  \item Set $k:=k+1$, go back to \ref{b_start}.
\end{enumerate}

\end{myalg}
The equivalence can be seen directly by computing the first-order optimality conditions. For a solution $u_k$ given by algorithm \ref{alg:Min2} we obtain
$$\big( S^\ast(Su_k - z) + \alpha_k ( u_k - \lambda_{k-1}), v-u_{k} \big) \geq 0, \quad \forall v \in \Uad, $$
while for an iterate $\bar u_k$ and resulting $\bar \mu_k$ of algorithm \ref{alg:Min3} we get
$$\big( S^\ast(S\bar u_k - z - \alpha_k \bar \mu_{k-1}) + \alpha_k \bar u_k, v-\bar u_{k} \big) \geq 0, \quad \forall v \in \Uad.$$
By adding both inequalities and applying an induction, we obtain
$$\|S(u_k - \bar u_k)\|_Y^2 + \alpha_k \|u_k - \bar u_k\|^2 \leq (\alpha_k S^\ast \mu_{k-1} - \alpha_k \lambda_{k-1}, \bar u_k - u_k) .$$
By definition $\lambda_{k-1} = S^\ast \mu_{k-1}$ and therefore both algorithms coincide.

\subsection{A priori error estimates for $H(u_k)$}
We want to show first error estimates in terms of $|H(u_k) - H(u^\dagger)|$, where $u^\dagger$ is a solution of \eqref{eq:2_main_problem}. The following result can be proven similar to the proof presented in \cite{osher2005} and is omitted here.

\begin{lemma}\label{lemma:functional_decreasing}
The iterates of algorithm \ref{alg:Min2} satisfy
$$H(u_k) \leq H(u_{k-1}).$$
\end{lemma}

Following the proof of \cite[Theorem 3.3]{osher2005} we can formulate a convergence result on $(H(u_k))_k$, together with an a-priori error estimate.

\begin{theorem}\label{thm:H_convergence}
The iterates of algorithm \ref{alg:Min2} satisfy
$$|H(u_k) - H(u^\dagger)| = \mathcal{O} \left( \gamma_k^{-1} \right).$$
Hence we have convergence, since the $\alpha_k$ are uniformly bounded. Furthermore we have
$$D^{\lambda_k}(u^\dagger, u_{k}) \leq D^{\lambda_{k-1}}(u^\dagger, u_{k-1}) \quad \text{and} \quad \sum\limits_{i=1}^\infty D^{\lambda_{i-1}}(u_i,u_{i-1})  < \infty.$$
\end{theorem}

The monotonicity of $D^{\lambda_k}(u^\dagger, u_{k})$ will play a crucial role in the subsequent analysis. Together with the lower bound \eqref{eq202}
it will allow to proof strong convergence $u_k\to u^\dagger$ under suitable conditions.

\subsection{Auxiliary estimates}

In the sequel, we will denote by $(u_k)_k$ the sequence of iterates provided by algorithm \ref{alg:Min2}.
Let us start with the following result, which will be useful in the convergence analysis later on.
\begin{lemma}\label{lemma:gamma_sum}
Let $\beta_j \geq 0$, such that $\beta_j \to 0$. We then have
$$\lim\limits_{k \to \infty}\gamma_k^{-1} \sum\limits_{j=1}^k \alpha_j^{-1} \beta_j = 0.$$
\end{lemma}

\begin{proof}

Let $\eps > 0$ be arbitrary. Since $\beta_j \to 0$ we can choose $N$ such that $\beta_j \leq \frac{\eps}{2}$ holds for all $j \geq N$. Since $\gamma_k^{-1} \to 0$ there is $M>N$ such that
$$\gamma_k^{-1} \sum\limits_{j=1}^N \alpha_j^{-1} \beta_j \leq \frac{\eps}{2}$$
holds for all $k \geq M$. We compute for $k \geq M$:
\begin{align*}
\gamma_k^{-1} \sum\limits_{j=1}^k \alpha_j^{-1} \beta_j &= \gamma_k^{-1} \sum\limits_{j=1}^N \alpha_j^{-1} \beta_j + \gamma_k^{-1} \sum\limits_{j=N+1}^k \alpha_j^{-1} \beta_j\\
&\leq \frac{\eps}{2} + \frac{\eps}{2} \gamma_k^{-1} \sum\limits_{j=N+1}^k \alpha_j^{-1} \leq \frac{\eps}{2} + \frac{\eps}{2} \gamma_k^{-1} \gamma_k
\leq \eps,
\end{align*}
which is the claim.
\end{proof}

In the case that $Su_k$ is equal to the optimal state $y^\dagger = Su^\dagger$,
the algorithm gives $u_{k+1}=u_k$, which is then a solution of \eqref{eq:2_main_problem}.

\begin{lemma}\label{lemma:algo_stat}
Let $y^\dagger$ be the optimal state of \eqref{eq:2_main_problem}. If $Su_k = y^\dagger$ then it holds $u_{k+1} = u_k$,
and $u_k$ solves \eqref{eq:2_main_problem}.
\end{lemma}

\begin{proof}
Since $u_{k+1}$ is the minimizer of
$$\frac{1}{2}\|Su - z\|_Y^2 + \alpha_{k+1} D^{\lambda_k}(u,u_k)$$
it follows
\begin{align*}
\frac{1}{2}\|Su_{k+1} - z\|_Y^2 + \alpha_{k+1} D^{\lambda_k}(u_{k+1},u_k) &\leq \frac{1}{2}\|Su_k - z\|_Y^2 + \alpha_{k+1} D^{\lambda_k}(u_k,u_k)\\
& = \frac{1}{2}\|y^\dagger - z\|_Y^2.
\end{align*}
Since $y^\dagger$ is the optimal state of \eqref{eq:2_main_problem}, it follows $\|y^\dagger - z\|_Y \le \|Su_{k+1} - z\|_Y$,
and hence we obtain
$$0 = D^{\lambda_k}(u_{k+1},u_k) = \frac{1}{2}\|u_{k+1} - u_k\|^2 - (w_k, u_{k+1} - u_k).$$
By construction we have $w_k \in \partial I_\Uad(u_k)$, so
$$\frac{1}{2}\|u_{k+1} - u_k\|^2 = (w_k, u_{k+1} - u_k) \leq 0,$$
which implies  $u_{k+1} = u_k$.
Since $Su_k=y^\dagger$ it follows that $u_k=u_{k+1}$ solves \eqref{eq:2_main_problem}.
\end{proof}

If the algorithm reaches a solution of \eqref{eq:2_main_problem}
after a finite number of steps, we can show that this solution satisfies a source condition.
This  condition is used below in Section \ref{sec:sc} to prove strong convergence of the iterates.

\begin{lemma}\label{lemma:Solution_character}
Let $u_k$ be a solution of \eqref{eq:2_main_problem} for some $k$. Then there exists a $w \in Y$ such that $u_k = P_\Uad(S^\ast w)$ holds.
\end{lemma}

\begin{proof}
For $k=0$ this is true by the definition of $u_0 = P_\Uad(0) = P_\Uad(S^\ast(0))$. For $k \geq 1$ we obtain with the optimality condition
$$u_k = P_\Uad(\lambda_k)  = P_\Uad(S^\ast \mu_k),$$
which is the stated result.
\end{proof}

Let us now prove auxiliary results that exploits the choice of the subdifferential $\lambda_k$ in \eqref{def:subdiff}. They will be employed in the convergence rate estimates below.

\begin{lemma}\label{lemma:technique_proof_start}
Let $u^\dagger$ be a solution of \eqref{eq:2_main_problem}.
Then it holds
\begin{multline}\label{eq308}
\frac{1}{\alpha_k} D^{\lambda_k}(u^\dagger , u_k) + \frac{1}{2\,\alpha_k^2}\|S(u^\dagger - u_k)\|_Y^2 + \frac{1}{2}\|v_{k}\|_Y^2\\
\leq \frac{1}{\alpha_k}(u^\dagger, u^\dagger - u_k) + \frac{\gamma_k}{\alpha_k} ( p^\dagger ,  u_k - u^\dagger ) + \frac{1}{2}\|v_{k-1}\|_Y^2
\end{multline}
where $v_k$ is defined by
\begin{equation}\label{def:v_k}
v_k := \sum\limits_{i=1}^k \frac{1}{\alpha_i} S(u^\dagger - u_i).
\end{equation}
\end{lemma}

\begin{proof}
First notice that $u^\dagger \in \partial J(u^\dagger)$ holds, which follows from
$$u^\dagger = u^\dagger + 0 \in \partial\left( \frac{1}{2}\|\cdot\|^2 \right)(u^\dagger) + \partial I_\Uad (u^\dagger) \subseteq \partial J(u^\dagger).$$
As in the proof of \cite[Theorem 4.1]{burger2007},
we consider the sum of the Bregman distances
\[
 \frac{1}{\alpha_k} D^{\lambda_k}(u^\dagger, u_k) + \frac{1}{\alpha_k} D^{u^\dagger}(u_k, u^\dagger) = \frac{1}{\alpha_k}( u^\dagger - \lambda_k, u^\dagger - u_k ).
\]
Using the definitions of $v_k$ and $p^\dagger$, we obtain
\begin{align*}
  \frac{1}{\alpha_k}(  - \lambda_k, u^\dagger - u_k )
&=  \frac{1}{\alpha_k} \left( \sum\limits_{j=1}^k \frac{1}{\alpha_j}(Su_j - z)  , S(u^\dagger - u_k)   \right)\\
&=  \frac{1}{\alpha_k} \left( \sum\limits_{j=1}^k \frac{1}{\alpha_j}(S(u_j - u^\dagger + u^\dagger) - z)  , S(u^\dagger - u_k)   \right)\\
&=  (-v_k,v_k-v_{k-1}) + \frac{1}{\alpha_k}\sum\limits_{j=1}^k \frac{1}{\alpha_j} ( Su^\dagger - z, S(u^\dagger - u_k) )\\
&=  (-v_k, v_k - v_{k-1}) + \frac{\gamma_k}{\alpha_k} ( p^\dagger ,  u_k - u^\dagger ).
\end{align*}
We continue with transforming the first addend on the right-hand side
\[
\begin{split}
 (-v_{k}, v_{k}-v_{k-1}) &= \frac{1}{2}\|v_{k-1}\|_Y^2 - \frac{1}{2}\|v_{k}\|_Y^2 - \frac{1}{2}\|v_{k}-v_{k-1}\|_Y^2\\
 &= \frac{1}{2}\|v_{k-1}\|_Y^2 - \frac{1}{2}\|v_{k}\|_Y^2 - \frac{1}{2\alpha_k^2}\|S(u^\dagger - u_k)\|_Y^2.
 \end{split}
\]
We obtain the result by using the nonnegativity of $D^{u^\dagger}(u_k, u^\dagger)$.
\end{proof}

Estimate \eqref{eq308} will play a key role in the convergence analysis of the algorithm. The principal idea is to sum the inequality \eqref{eq308}
with respect to $k$. Using the monotonicity of the Bregman distance $D^{\lambda_k}(u^\dagger , u_k)$ and inequality \eqref{eq202}, we can then conclude convergence of the iterates
if we succeed in estimating the terms involving the scalar product $(u^\dagger, u^\dagger - u_k)$. Note that due to Theorem \ref{thm:relation_ud_pd} the term $( p^\dagger ,  u_k - u^\dagger )$
is non-positive.

\section{Convergence of the Bregman iteration}
\label{sec4}

In this section we  study  convergence of the iterates $(u_k)_k$ of algorithm \ref{alg:Min2}.

\subsection{General convergence results}
First we present a general convergence result.

\begin{theorem}\label{thm:weak_limit_points}
Weak limit points of the sequence $(u_k)_k$ generated by algorithm \ref{alg:Min2} are solutions to the problem \eqref{eq:2_main_problem}. Furthermore the iterates satisfy
$$\sum\limits_{i=1}^\infty \|u_{i}-u_{i-1}\|^2 < \infty.$$
\end{theorem}

\begin{proof}
Since $L^2(\Omega)$ is a Hilbert space and $\Uad$ is bounded, closed and convex, it is weakly relatively compact and weakly closed. Hence we can deduce the existence of a subsequence $u_{k_j} \rightharpoonup u^\ast \in \Uad$. Furthermore $H$ is convex and continuous, so it is weakly lower semi-continuous. By Theorem \ref{thm:H_convergence} we know that the sequence $(H(u_k))_k$ is converging towards $H(u^\dagger)$, hence we obtain
$$H(u^\dagger)= \liminf\limits_{j \to \infty} H(u_{k_j}) \geq H(u^\ast),$$
yielding $H(u^\dagger) = H(u^\ast)$, since $u^\dagger$ realizes the minimum of $H$ in $\Uad$. So $u^\ast$ is a solution to the problem \eqref{eq:2_main_problem}.
To prove the second part we use \eqref{eq202} and the result of Theorem \ref{thm:H_convergence}
to show
\[
\sum\limits_{i=1}^\infty \frac{1}{2}\|u_{i}-u_{i-1}\|^2
\le \sum\limits_{i=1}^\infty D^{\lambda_{i-1}}(u_i,u_{i-1}) < \infty,
\]
which ends the proof.
\end{proof}

\begin{remark}
The above result resembles properties of the iterates generated by the PPM.
There it holds $\sum_{i=1}^\infty \|u_{i}-u_{i-1}\|^2 < \infty$, see e.g. \cite{Solodov00}.
\end{remark}

As argued in Section \ref{sec21}, the optimal state $y^\dagger$ of \eqref{eq:2_main_problem} is uniquely determined.
This allows to prove the strong convergence $(Su_k)$ under mild conditions on the parameters $\alpha_k$.

\begin{theorem}
Let the sequence $(u_k)_k$ be generated by algorithm \ref{alg:Min2}. Then it holds \[Su_k\to y^\dagger,\]
where $y^\dagger$ is the uniquely determined optimal state of \eqref{eq:2_main_problem}.
\end{theorem}
\begin{proof}
Let $(u_{k'})_{k'}$ be a subsequence of the sequence of iterates. Due to the boundedness of $U_{ad}$, this sequence is bounded, and has
a weakly converging subsequence $(u_{k''})_{k''}$, $u_{k''}\to u^*$. By Theorem \ref{thm:weak_limit_points}, the limit $u^*$ is a solution of
\eqref{eq:2_main_problem}. This implies $Su^*=y^\dagger$.
Hence, we proved that each subsequence of $(Su_k)_k$ contains a subsequence that weakly converges to $y^\dagger$.
This shows $Su_k \rightharpoonup y^\dagger$.

Due to Theorem \ref{thm:H_convergence} and $\gamma_k^{-1}\to0$, we have that
\[
 H(u_k) = \frac12\|Su_k-z\|_Y^2 \to \frac12 \|y^\dagger-z\|_Y^2 = H(u^\dagger)
\]
for every solution $u^\dagger$ of \eqref{eq:2_main_problem}. This implies convergence of the norms $\|Su_k\|_Y\to \|y^\dagger\|_Y$.
Since $Y$ is a Hilbert space, the strong convergence $Su_k\to y^\dagger$ follows immediately.
\end{proof}

If we assume that the problem \eqref{eq:2_main_problem} has a unique solution $u^\dagger \in \Uad$ we can prove strong convergence of our algorithm.

As argued above, the solution of  \eqref{eq:2_main_problem} is uniquely determined if,
e.g., the operator $S$ is injective or $p^\dagger \neq 0$ almost everywhere.

\begin{theorem}\label{thm:unique_weak_conv}
Assume that $u^\dagger \in \Uad$ is the unique solution of \eqref{eq:2_main_problem}. Then the iterates of algorithm \ref{alg:Min2} satisfy
$$ \lim\limits_{k \to \infty} \| u_k - u^\dagger\| = 0 \quad \text{and} \quad \min\limits_{j=1,...,k} \frac{1}{\alpha_j}\|S(u_j -u^\dagger)\|_Y^2 \to 0.$$
\end{theorem}

\begin{proof}
With Theorem \ref{thm:weak_limit_points} we know that each weak limit point is a solution to the problem \eqref{eq:2_main_problem}. So let $u^\ast$ be such a point which satisfy $H(u^\dagger) = H(u^\ast)$. As $u^\dagger$ is the unique solution we conclude $u^\ast = u^\dagger$. From every subsequence of $(u_k)_k$ we can extract a weakly converging subsequence and repeat this argumentation. Hence we can conclude weak convergence
$u_k \rightharpoonup u^\dagger$ of the whole sequence.

With Lemma \ref{lemma:technique_proof_start} and Theorem \ref{thm:relation_ud_pd} we obtain
\[
\frac{1}{2\,\alpha_k^2}\|S(u^\dagger - u_k)\|_Y^2 + \frac{1}{\alpha_k} D^{\lambda_k}(u^\dagger, u_k)+\frac{1}{2}\|v_k\|_Y^2 \leq \frac{1}{\alpha_k}(u^\dagger, u^\dagger - u_k) + \frac{1}{2}\|v_{k-1}\|_Y^2 .
\]
Summing up yields
\[
\sum \limits_{j=1}^k \frac{1}{2\,\alpha_j^2} \|S(u^\dagger - u_j)\|_Y^2 + \sum\limits_{j=1}^k \frac{1}{\alpha_j}  D^{\lambda_j}(u^\dagger, u_j) \leq \sum\limits_{j=1}^k \alpha_j^{-1} (u^\dagger , u^\dagger - u_j).
\]
where we used the convention $v_0=0$. We now use the monotonicity of $D^{\lambda_k}(u^\dagger, u_k)$ (see Theorem \ref{thm:H_convergence}) and the estimate
$\frac{1}{2} \|u^\dagger-u_k\|^2 \leq D^{\lambda_k}(u^\dagger , u_k)$
to obtain
$$\min\limits_{j=1,...,k} \frac{1}{\alpha_j} \|S(u^\dagger - u_j)\|_Y^2 + \|u^\dagger - u_k\|_Y^2 \leq 2\gamma_k^{-1} \sum\limits_{j=1}^k \alpha_j^{-1} (u^\dagger, u^\dagger - u_j).$$
We finally obtain the result by using the weak convergence $u_k \rightharpoonup u^\dagger$ and Lemma~\ref{lemma:gamma_sum}.
\end{proof}

\subsection{Strong convergence for the Source Condition}
\label{sec:sc}
A common assumption on a solution $u^\dagger$ is the following source condition, which is an abstract smoothness condition (see \cite{burger2007,neubauer1988,wachsmuth2011,wachsmuth2011b}). We say $u^\dagger$ satisfies the source condition \ref{ass:SC} if the following assumption holds.
{
\renewcommand{\thetheorem}{\textbf{SC}}
\begin{assumption}[Source Condition]\label{ass:SC}
Let $u^\dagger$ be a solution of \eqref{eq:2_main_problem}.
Assume that there exists an element $w \in Y$ such that $u^\dagger = P_\Uad(S^\ast w)$ holds.
\end{assumption}
}

The source condition is equivalent to the existence of Lagrange multipliers for the problem
\begin{equation}\label{eq:mini_norm}\begin{split}
\min\limits_{u \in \Uad} \quad & \frac{1}{2} \|u\|^2\\
\text{such that} \quad & Su = y^\dagger,
\end{split}
\end{equation}
where $y^\dagger$ is the uniquely defined optimal state of \eqref{eq:2_main_problem}.
To see this, consider the Lagrange function
$$\mathcal{L}(u,w) := \frac{1}{2}\|u\|^2 + (w, y^\dagger - Su).$$
For every $u^\dagger$ satisfying $Su^\dagger = y^\dagger$ we obtain
$$\frac{\partial}{\partial w} \mathcal{L}(u^\dagger, w^\dagger) = y^\dagger - Su^\dagger = 0.$$
This means, the function $w^\dagger$ is a Lagrange multiplier if and only if:
\begin{align*}
&\frac{\partial}{\partial u} \mathcal{L}(u^\dagger, w^\dagger) (v-u^\dagger) \geq 0 \quad \forall v \in \Uad\\
\iff & (u^\dagger - S^\ast w^\dagger , v-u^\dagger) \geq 0\quad \forall v \in \Uad\\
\iff & u^\dagger = P_\Uad(S^\ast w^\dagger)
\end{align*}
Hence, if the control $u^\dagger$ satisfies \ref{ass:SC} then it is a solution of \eqref{eq:mini_norm}.
Moreover, as this optimization problem is uniquely solvable, it follows
that there is at most one control satisfying \ref{ass:SC}.
Note that the existence of Lagrange multipliers is not guaranteed in general, as
in many situations the operator $S$ is compact and has non-closed range.\\

Under this assumption we can prove strong convergence of algorithm \ref{alg:Min2}.

\begin{theorem}\label{thm:SC_strong_conv}
Assume that Assumption \ref{ass:SC} holds for $u^\dagger$. Then the iterates of algorithm \ref{alg:Min2} satisfy
\begin{align*}
\|u_k - u^\dagger\|^2 &= \mathcal{O}(\gamma_k^{-1})\\
\min\limits_{i=1,...,k} \|S(u_i - u^\dagger)\|_Y^2 &= \mathcal{O} \left( \left(\sum_{i=1}^k \alpha_i^{-2}\right)^{-1} \right).
\end{align*}
\end{theorem}

\begin{proof}
From Lemma \ref{lemma:technique_proof_start} we know
\[
\frac{1}{\alpha_k} D^{\lambda_k}(u^\dagger , u_k) + \frac{1}{2\,\alpha_k^2}\|S(u^\dagger - u_k)\|_Y^2 + \frac{1}{2}\|v_{k}\|_Y^2
\leq \frac{1}{\alpha_k}(u^\dagger, u^\dagger - u_k) + \frac{1}{2}\|v_{k-1}\|_Y^2.
\]
It remains to estimate $(u^\dagger, u^\dagger - u_k)$ with the help of the source condition.
By the definition of the projection $u^\dagger = P_\Uad(S^\ast w)$ we get
$$\big(u^\dagger- S^\ast w, v-u^\dagger\big) \geq 0 \quad \forall v \in \Uad.$$
Since $u_{k} \in \Uad$ we have
\[
\frac{1}{\alpha_k}\big(u^\dagger, u^\dagger - u_{k}\big) \leq \frac{1}{\alpha_k}\big(S^\ast w, u^\dagger - u_{k}\big)
= \frac{1}{\alpha_k} (w, S(u^\dagger-u_k))_Y = (w, v_k-v_{k-1}).
\]
Plugging this in the estimate above yields
\[
\frac{1}{\alpha_k} D^{\lambda_k}(u^\dagger , u_k) + \frac{1}{2\,\alpha_k^2}\|S(u^\dagger - u_k)\|_Y^2 + \frac{1}{2}\|v_{k}-w\|_Y^2
\leq   \frac{1}{2}\|v_{k-1}-w\|_Y^2.
\]
Following the lines of Theorem \ref{thm:unique_weak_conv} we obtain by a summation
$$\frac{1}{2} \sum \limits_{j=1}^k \frac{1}{\alpha_j^2} \|S(u^\dagger - u_j)\|_Y^2 + \frac{\gamma_k}{2} \|u^\dagger - u_k\|^2 + \frac{1}{2}\|v_k - w\|_Y^2 \leq \frac{1}{2}\|w\|_Y^2,$$
which yields the result.
\end{proof}

Under the source condition \ref{ass:SC} we can improve Lemma \ref{lemma:algo_stat}.
\begin{lemma}
Assume that $u^\dagger$ satisfies Assumption \ref{ass:SC}. If it holds $Su_k = y^\dagger$, then it follows $u_k = u^\dagger$.
\end{lemma}

\begin{proof}
As argued in Lemma \ref{lemma:Solution_character}, $u_k$ fulfills \ref{ass:SC}.
Hence both $u_k$ and $u^\dagger$ are solutions of the minimal norm problem \ref{eq:mini_norm}.
This problem is uniquely solvable, which yields $u_k=u^\dagger$.
\end{proof}

While the sequence $(\lambda_k)_k$ is unbounded in general, we can prove convergence of $\gamma_k^{-1}\lambda_k$,
which is a weighted average of the sequence $\big(S^*(z-Su_k)\big)_k$.

\begin{corollary}\label{coro410}
Assume that Assumption \ref{ass:SC} holds for $u^\dagger$.
Then it holds
\[
\left\|\gamma_k^{-1}\sum_{i=1}^k \frac{1}{\alpha_i} S(u_i-u^\dagger)\right\|_Y^2 +
\left\|\gamma_k^{-1} \lambda_k - p^\dagger \right\|^2 = \mathcal O(\gamma_k^{-2}) .
\]
\end{corollary}
\begin{proof}

Due to the definitions of $\lambda_k$, $p^\dagger$, and $v_k$ it holds
 \[
  \gamma_k^{-1} \lambda_k - p^\dagger = \gamma_k^{-1} \left( \sum_{i=1}^k \frac{1}{\alpha_i} S^*S(u^\dagger-u_i)
  \right) = \gamma_k^{-1} S^* v_k.
 \]
Following the lines of the proof of Theorem \ref{thm:SC_strong_conv}, we obtain
\[
 \|v_k\|_Y \le \|v_k-w\|_Y + \|w\|_Y \le 2\|w\|_Y,
\]
which yields the claim.
\end{proof}

When comparing the convergence rates of Theorem \ref{thm:SC_strong_conv} and Corollary~\ref{coro410},
one sees that the norm of the weighted average $\gamma_k^{-1}\sum_{i=1}^k \frac{1}{\alpha_i} S(u_i-u^\dagger)$ converges faster to zero than $\min\limits_{i=1,...,k} \|S(u_i - u^\dagger)\|_Y$,
since it holds
$
 \gamma_k^2 = \left( \sum_{i=1}^k \alpha_i^{-1}\right)^2 > \sum_{i=1}^k \alpha_i^{-2}.
$

\subsection{Convergence results for the Active Set Condition}
\label{sec44}
If $z$ is not attainable, i.e., $y^\dagger\ne z$, a solution $u^\dagger$ may be bang-bang, i.e., $u^\dagger$ is a linear combination of characteristic functions, hence discontinuous in general with $u^\dagger \not\in H^1(\Omega)$. But in many examples the range of $S^\ast$ contains $H^1(\Omega)$ or $C(\bar \Omega)$. Hence, the source condition \ref{ass:SC} is too restrictive for bang-bang solutions.
We will thus resort to the following condition.
We say that $u^\dagger$ satisfies the active set condition \ref{ass:ActiveSet}, if the following assumption holds.
Let us recall the definition of $p^\dagger= S^\ast(z- Su^\dagger)$.
{
\renewcommand{\thetheorem}{\textbf{ASC}}
\begin{assumption}[Active Set Condition]\label{ass:ActiveSet}
Let $u^\dagger$ be a solution of \eqref{eq:2_main_problem} and assume that there exists a set $I \subseteq \Omega$, a function $w \in Y$, and positive constants $\kappa, c$ such that the following holds

\begin{enumerate}
  \item (source condition) $I \supset \{ x \in \Omega: \; p^\dagger(x) = 0 \}$  and
  $$\chi_I u^\dagger = \chi_I P_\Uad (S^\ast w),$$
  \item (structure of active set) $A := \Omega \setminus I$ and for all $\eps > 0$
  $$|\{ x\in A: \; 0 < |p^\dagger(x)| < \eps  \}| \leq c \eps^\kappa,$$
  \item (regularity of solution) $S^\ast w \in L^\infty(\Omega)$.
\end{enumerate}
\end{assumption}
}

\begin{remark}
Following \cite[Remark 3.1]{wachsmuth2011}, there exists at most one $u^\dagger \in \Uad$ satisfying Assumption \ref{ass:ActiveSet}. Furthermore by \cite[Remark 3.1]{wachsmuth2011} this has to be the minimal norm solution in $\Uad$, which is unique by \cite[Lemma 2.3]{wachsmuth2011}.
\end{remark}

This condition is used in \cite{wachsmuth2011}.
It was applied for the case $\kappa = 1$, $I = \emptyset$ and $A = \Omega$ in \cite{wachsmuth2011b}. The set $I$ contains the set $\{x \in \Omega: \; p^\dagger(x) = 0 \}$, which is the set of points where $u^\dagger(x)$ cannot be uniquely determined from $p^\dagger(x)$, compare to Theorem \ref{thm:relation_ud_pd}. On this set, we assume that $u^\dagger$ fulfills a local source condition, which implies that $u^\dagger$ has some extra regularity there. The set $A$ contains the points, where the inequality constraints are active, since it holds by construction that $p^\dagger(x) \neq 0$ on $A$, which implies $u^\dagger(x) \in \{ u_a(x), u_b(x) \}$.

In the following we will show convergence results for iterates produced by algorithm \ref{alg:Min2} if we assume \ref{ass:ActiveSet}. The special case $I = \Omega$ is already covered by Theorem \ref{thm:SC_strong_conv}, since for this choice of $I$ the Assumption \ref{ass:ActiveSet} reduces to the Assumption \ref{ass:SC}.

We now focus on the case $I \neq \Omega$, that is, if the source condition is not satisfied on the whole domain $\Omega$.

At first, let us prove a strengthened version of the first-order optimality conditions satisfied
by $u^\dagger$. We refer to \cite[Lemma 1.3]{seydenschwanz2015} for a different proof.

\begin{lemma}\label{lem31}
Let $u^\dagger$ satisfy Assumption \ref{ass:ActiveSet}. Then there is $c_A>0$ such that for all $u\in U_{ad}$
 \[
  (-p^\dagger, u-u^\dagger) \ge c_A\|u-u^\dagger\|_{L^1(A)}^{1+\frac1\kappa}
 \]
is satisfied.
\end{lemma}
\begin{proof}
Let $\eps>0$ be given. Let us define $A_\eps:=\{ x\in A:\ |p^\dagger(x)| \ge  \eps\}$.
Then it holds
\[\begin{split}
-\int\limits_{\Omega} p^\dagger(u-u^\dagger) &\ge- \int\limits_{A_\eps} p^\dagger(u-u^\dagger) - \int\limits_{A\setminus A_\eps}p^\dagger(u-u^\dagger)\\
  & \ge \eps \,\|u-u^\dagger\|_{L^1(A_\eps)} - \eps \,\|u-u^\dagger\|_{L^1(A\setminus A_\eps)}.
  \end{split}\]
Using Assumption \ref{ass:ActiveSet} to estimate the measure of the set $A\setminus A_\eps$ we proceed with
\begin{align*}
  \eps \,\|u-u^\dagger\|_{L^1(A_\eps)} &- \eps \,\|u-u^\dagger\|_{L^1(A\setminus A_\eps)} \\
  &\ge \eps \,\|u-u^\dagger\|_{L^1(A)} - 2\,\eps \,\|u-u^\dagger\|_{L^1(A\setminus A_\eps)}\\
  &\ge \eps \,\|u-u^\dagger\|_{L^1(A)} - 2\, \eps \,\|u-u^\dagger\|_{L^\infty(A)}\, |A\setminus A_\eps|\\
  &\ge \eps \,\|u-u^\dagger\|_{L^1(A)} - c\, \eps^{\kappa+1},
\end{align*}
where $c>1$ is a constant independent of $u$.
In the last step, we used that the control bounds are given in $L^\infty(\Omega)$.
Setting $\eps:=c^{-2/\kappa}\|u-u^\dagger\|_{L^1(A)}^{1/\kappa}$ yields
\[
  (-p^\dagger, u-u^\dagger) \ge c \|u-u^\dagger\|_{L^1(A)}^{1+\frac1\kappa},
\]
which is the claim.
\end{proof}

The next step concerns the estimation of $(u^\dagger, u^\dagger - u_j)$ with the help of the source condition part of  \ref{ass:ActiveSet}.

\begin{lemma}\label{lem414}
Let $u^\dagger$ satisfy \ref{ass:ActiveSet}.
If $I \neq \Omega$ there is a constant $c>0$ such that for all $k$ it holds
\[
(u^\dagger, u^\dagger - u_k) \leq (S^\ast w, u^\dagger - u_k) + c\ \|u^\dagger - u_k\|_{L^1(A)}.
\]
\end{lemma}

\begin{proof}
Since $U_{ad}$ is defined by pointwise inequalities, the projection onto $U_{ad}$
can be taken pointwise.
This implies
$$\big(\chi_I(u^\dagger - S^\ast w), v-u^\dagger\big) \geq 0, \quad \forall v \in \Uad,$$
leading to
$$(\chi_I u^\dagger, u^\dagger - u_k) \leq ( \chi_I S^\ast w, u^\dagger - u_k ).$$
This gives
\begin{align*}
(u^\dagger, u^\dagger - u_k) &= (\chi_I u^\dagger + \chi_A u^\dagger, u^\dagger - u_k)\\
&\leq ( \chi_I S^\ast w + \chi_A u^\dagger, u^\dagger - u_k )\\
&= \big(S^\ast w, \chi_I(u^\dagger - u_k)\big) + (\chi_A u^\dagger, u^\dagger - u_k).
\end{align*}
Since $\chi_I = 1 - \chi_A$ we have
$$S \chi_I (u^\dagger - u_k) = S(1- \chi_A)(u^\dagger - u_k) = S(u^\dagger - u_k) - S \chi_A (u^\dagger - u_k).$$
Hence
\begin{align*}
(u^\dagger, u^\dagger - u_k) &\leq \big(w, S(u^\dagger - u_k) - S\chi_A (u^\dagger - u_k)\big) + \big(u^\dagger, \chi_A(u^\dagger - u_k)\big)\\
&= \big(w, S(u^\dagger - u_k)\big) + \big(u^\dagger - S^\ast w, \chi_A (u^\dagger - u_k)\big).
\end{align*}

Since on $A$ we have $p^\dagger \neq 0$ and $u^\dagger \in L^\infty(A)$, (recall $u_a,u_b \in L^\infty(A)$) so using the regularity assumption $S^\ast w \in L^\infty(\Omega)$ we can estimate
$$\big(u^\dagger - S^\ast w, \chi_A (u^\dagger - u_k)\big) \leq c \|u^\dagger - u_k\|_{L^1(A)},$$
which is the claim.
\end{proof}

We now have all the tools to prove strong convergence for the iterates of Algorithm \ref{alg:Min2}.

\begin{theorem}\label{thm:ASC_strong_convergence}
Let $u^\dagger$ satisfy Assumption \ref{ass:ActiveSet}. Then the iterates of Algorithm \ref{alg:Min2} satisfy
\begin{align*}
\|u^\dagger - u_k\|^2 &= \mathcal{O}\left( \gamma_k^{-1} + \gamma_k^{-1} \sum\limits_{j=1}^k \alpha_j^{-1} \gamma_j^{- \kappa}   \right),\\
\min\limits_{j=1,...,k} \|S(u^\dagger - u_j)\|_Y^2 &= \mathcal{O} \left(  \left( \sum\limits_{j=1}^k \frac{1}{\alpha_j^2} \right)^{-1} \left( 1 + \sum\limits_{j=1}^k \alpha_j^{-1} \gamma_j^{- \kappa}   \right)  \right),\\
\min\limits_{j=1,...,k} \|u^\dagger - u_j\|_{L^1(A)}^{1+ \frac{1}{\kappa}} &= \mathcal{O} \left(  \left( \sum\limits_{j=1}^k \frac{\gamma_j}{\alpha_j} \right)^{-1} \left( 1 + \sum\limits_{j=1}^k \alpha_j^{-1} \gamma_j^{- \kappa}   \right)  \right).
\end{align*}
\end{theorem}

\begin{proof}
Using the results of Lemmas \ref{lemma:technique_proof_start}, \ref{lem31}, and  \ref{lem414} we obtain
\begin{multline*}
\frac{1}{\alpha_k} D^{\lambda_k}(u^\dagger , u_k) + \frac{1}{2\,\alpha_k^2}\|S(u^\dagger - u_k)\|_Y^2 + \frac{1}{2}\|v_{k}\|_Y^2-  \frac{1}{2}\|v_{k-1}\|_Y^2\\
\begin{aligned}
&\leq \frac{1}{\alpha_k}(u^\dagger, u^\dagger - u_k) + \frac{\gamma_k}{\alpha_k} ( p^\dagger ,  u_k - u^\dagger )\\
&\le  \frac{1}{\alpha_k}(S^\ast w, u^\dagger - u_k) + \frac{c}{\alpha_k}\ \|u^\dagger - u_k\|_{L^1(A)} - \frac{c_A \gamma_k}{\alpha_k} \|u^\dagger - u_k\|_{L^1(A)}^{1 + \frac{1}{\kappa}}\\
&\le  ( w, v_k-v_{k-1}) + \frac{c}{\alpha_k}\ \|u^\dagger - u_k\|_{L^1(A)} - \frac{c_A \gamma_k}{\alpha_k} \|u^\dagger - u_k\|_{L^1(A)}^{1 + \frac{1}{\kappa}}.
\end{aligned}
\end{multline*}
By Young's inequality, we find
\[
 \frac{c}{\alpha_k}\ \|u^\dagger - u_k\|_{L^1(A)} \le \frac{c_A \gamma_k}{2\,\alpha_k} \|u^\dagger - u_k\|_{L^1(A)}^{1 + \frac{1}{\kappa}}
 +c\, \frac{\gamma_k^{-\kappa}}{\alpha_k}.
\]
This implies the estimate
\begin{multline*}
\frac{1}{\alpha_k} D^{\lambda_k}(u^\dagger , u_k) + \frac{1}{2\,\alpha_k^2}\|S(u^\dagger - u_k)\|_Y^2 + \frac{c_A \gamma_k}{2\,\alpha_k} \|u^\dagger - u_k\|_{L^1(A)}^{1 + \frac{1}{\kappa}} + \frac{1}{2}\|v_{k}-w\|_Y^2\\
\le  \frac{1}{2}\|v_{k-1}-w\|_Y^2+c\, \frac{\gamma_k^{-\kappa}}{\alpha_k}.
\end{multline*}
Summation of this inequality together with the monotonicity of the Bregman distance gives
\begin{align*}
\sum\limits_{j=1}^k \frac{1}{\alpha_j^2}\|S(u^\dagger - u_j)\|_Y^2 &+ \sum\limits_{j=1}^k \frac{\gamma_j}{\alpha_j} \|u^\dagger - u_j\|_{L^1(A)}^{1 + \frac{1}{\kappa}}\\
&+ \gamma_k  D^{\lambda_k} (u^\dagger , u_k) + \|v_k - w\|_Y^2  \leq c \left( 1 + \sum\limits_{j=1}^k \alpha_j^{-1} \gamma_j^{-\kappa} \right).
\end{align*}
The claim now follows using the lower bound \eqref{eq202}.
\end{proof}

If assumption \ref{ass:ActiveSet} is satisfied with $A=\Omega$, which implies that $u^\dagger$ is bang-bang on $\Omega$, or $w=0$,
then the estimate of Theorem \ref{thm:ASC_strong_convergence} can be improved to
\[
 \|u^\dagger - u_k\|^2 \le c\  \gamma_k^{-1} \sum\limits_{j=1}^k \alpha_j^{-1} \gamma_j^{-\kappa}.
\]

Similar to Corollary \ref{coro410} we can prove convergence of the weighted average $\gamma_k^{-1}\lambda_k$.
\begin{corollary}\label{cor:ASClambdaestimate}
Let $u^\dagger$ satisfy \ref{ass:ActiveSet}.
Then it holds
\[
\left\|\gamma_k^{-1}\sum_{i=1}^k \frac{1}{\alpha_i} S(u_i-u^\dagger)\right\|_Y^2 +
\left\|\gamma_k^{-1} \lambda_k - p^\dagger \right\|^2 = \mathcal O\left(\gamma_k^{-2}\left( 1   +  \sum_{j=1}^k \alpha_j^{-1} \gamma_j^{-\kappa}\right)\right) .
\]
\end{corollary}

\begin{proof}
Following the lines of theorem \ref{thm:ASC_strong_convergence} we obtain
\[
\|v_k\|_Y^2
\leq c(\|v_k - w\|_Y^2 + \|w\|_Y^2) \leq   c\left( 1   +  \sum_{j=1}^k \alpha_j^{-1} \gamma_j^{-\kappa}\right).
\]
The claim follows with the same arguments as in Corollary \ref{coro410}.
\end{proof}

Let us derive precise convergence rates, if $\alpha_k$ is a polynomial in $k$.
\begin{corollary}\label{coro416}
Let $u^\dagger$ satisfy \ref{ass:ActiveSet}.
Suppose that $\alpha_k$ is given by $\alpha_k = c_\alpha k^{-s}$ with $s\ge0$, $c_\alpha>0$.
Then it holds
\begin{multline*}
  k^{s+1} \|u^\dagger - u_k\|^2
  + k^{2(s+1)}\min\limits_{j=1,...,k} \|u^\dagger - u_j\|_{L^1(A)}^{1+\frac1\kappa}\\
  + k^{2s+1}  \min\limits_{j=1,...,k} \|S(u^\dagger - u_j)\|_Y^2
  +k^{2(s+1)} \|\gamma_k^{-1} \lambda_k - p^\dagger\|_Y^2
\\
\leq c \begin{cases} k^{(s+1)(1 - \kappa)} & \text{ if } \kappa <1,\\
 \log(k) & \text{ if } \kappa =1,\\
 1 & \text{ if } \kappa > 1.
 \end{cases}
\end{multline*}
\end{corollary}
\begin{proof}
For this choice of $\alpha_k$, it is easy to see that $\gamma_k^{-1} \leq c k^{-(s+1)}$. Then $\alpha_j^{-1} \gamma_j^{-\kappa} \leq c j^{s-(s+1)\kappa}$ which implies that $\sum_{j=1}^k \alpha_j^{-1} \gamma_j^{-\kappa} \leq c k^{(s+1)(1-\kappa)}$ if $\kappa\ne1$ and otherwise $\sum_{j=1}^k \alpha_j^{-1} \gamma_j^{-\kappa} \leq c \log(k)$ if $\kappa=1$.
If $\kappa\le 1$ then the term $\sum_{j=1}^k \alpha_j^{-1} \gamma_j^{-\kappa}$
is dominating the error estimate, while for $\kappa>1$ this term tends to zero.\\

This yields
\[
\|u^\dagger - u_k\|^2 \leq c\, \gamma_k^{-1} \left(1  +  \sum\limits_{j=1}^k \alpha_j^{-1} \gamma_j^{- \kappa}   \right)
\leq  c\, k^{-(s+1)}  s_k
 \]
with
\[
 s_k := \begin{cases} k^{(s+1)(1 - \kappa)} & \text{ if } \kappa <1,\\
 \log(k) & \text{ if } \kappa =1,\\
 1 & \text{ if } \kappa > 1.
 \end{cases}
\]
Using $\sum_{j=1}^k \alpha_j^{-1}\gamma_j \geq c k^{2(s+1)}$ and $\sum_{j=1}^k \alpha_j^{-2} \geq c k^{2s+1}$, we obtain the estimates
\begin{align*}
 \min\limits_{j=1,...,k} \|u^\dagger - u_j\|_{L^1(A)}^{1+\frac1\kappa}
 &\le c \left(\sum_{j=1}^k \alpha_j^{-1}\gamma_j\right)^{-1}\left( 1+  \sum\limits_{j=1}^k \alpha_j^{-1} \gamma_j^{-\kappa}  \right)\\
 &\leq c k^{-2(s+1)} s_k
\end{align*}
and
\begin{align*}
 \min\limits_{j=1,...,k} \|S(u^\dagger - u_j)\|_Y^2
 &\le c \left(\sum_{j=1}^k \alpha_j^{-2}\right)^{-1}\left( 1+  \sum\limits_{j=1}^k \alpha_j^{-1} \gamma_j^{-\kappa}  \right)\\
 &\leq c k^{-(2s+1)} s_k.
\end{align*}
Similar we obtain with Corollary \ref{cor:ASClambdaestimate}
\begin{align*}
 \|\gamma_k^{-1} \lambda_k - p^\dagger\|_Y^2
 &\le c \gamma_k^{-2}\left( 1+  \sum\limits_{j=1}^k \alpha_j^{-1} \gamma_j^{-\kappa}  \right)\\
 &\leq c k^{-2(s+1)} s_k.
\end{align*}
Combining these 4 inequalities yields the claim.
\end{proof}

\section*{Funding}
This work was supported by DFG under grant number Wa 3626/1-1.

\bibliographystyle{plain}
\bibliography{literatur}

\end{document}